\documentclass[12pt]{article} 
\usepackage{amsfonts,amsmath,latexsym,amssymb,mathrsfs,amsthm, color, comment}

\def\numberlikeadb{\global\def\theequation{\thesection.\arabic{equation}}}
\numberlikeadb
\newtheorem{theorem}{Theorem}[section]

\newtheorem{corollary}[theorem]{Corollary}

\newtheorem{remark}[theorem]{Remark}
\newtheorem{example}[theorem]{Example}

\numberwithin{equation}{section}

\begin{document}

\title{Poisson approximation of subgraph counts in stochastic block models and a graphon model}
\author{Matthew Coulson\footnote{The Queen's College, University of Oxford, High Street, OXFORD, OX1 4AW, UK},\,\, Robert E. Gaunt\footnote{Department of Statistics,
University of Oxford, 24-29 St Giles', OXFORD OX1 3LB, UK}\,\, and Gesine Reiner$\mathrm{t}^\dagger$
}

\date{\today} 
\maketitle

\begin{abstract} Small subgraph counts can be used as summary statistics for large random graphs. We use the Stein-Chen method to derive Poisson approximations for the distribution of the number of subgraphs in the stochastic block model which are isomorphic to some fixed graph.  We also obtain Poisson approximations for subgraph counts in a graphon-type generalisation of the model in which the edge probabilities are (possibly dependent) random variables supported on a subset of $[0,1]$.  Our results apply when the fixed graph is a member of the class of strictly balanced graphs. 
\end{abstract}


\noindent{{{\bf{AMS 2010 Subject Classification:}}}   90B15, 62E17, 60F05, 05C80

\section{Introduction}

Small subgraph counts can be used as summary statistics for large random graphs; indeed in some graph models they appear as sufficient statistics, see \cite{frank}. Moreover, many networks are conjectured to have over- or under-represented motifs (small subgraphs), see for example \cite{alon}.  Statistics based on small subgraph counts can be used to compare networks, as in  \cite{ali14, sarajlic13}.   To determine which small subgraphs are unusual, assessing the distribution of such motifs is key. While \cite{picard} gives the mean and variance for some common random graph models, \cite{picard} does not derive a distributional approximation. 

In this paper, we address the issue of such a distributional approximation for a large class of models which include stochastic block models and a graphon model but also models with random edge probabilities, provided that the edge probabilities display some local dependence, which will be made clearer in Section 3.

The stochastic block model (SBM) was introduced originally for directed graphs by \cite{holland} and generalised to other graphs by   \cite{ns01}; it is also called  Erd\H{o}s-R\'{e}nyi Mixture Model  in \cite{daubin}, {{and in theoretical computer science it is called the Planted Partition Model \cite{condonkarp}.}}  It has  a wide range of applications, see for example \cite{airoldi, bickelchen, daubin, karrernewman, olhedewolfe}, and \cite{matias} for a recent survey.  The  model is defined as follows.  Consider an undirected random graph on $n$ vertices, with no self-loops or multiple edges, in which the vertices are spread among $Q$ hidden classes with respective proportion vector $f= (f_1,\ldots,f_Q)$.  The class label of a vertex is drawn from a multinomial distribution ${\cal{M}}(1, f)$, and class assignments are independent of each other. Edges $Y_{i,j}$ are independent conditionally on the class of the vertices, and the edge probability depends only on the classes of the vertices:
\begin{equation*}\mathbb{P}(Y_{i,j}=1\,|\,i\in a,j\in b)=\pi_{a,b}.
\end{equation*}
We shall denote this model by $SBM(n,\pi,f)$.  If $\pi_{a,b}=p$ for all $a$ and $b$, the SBM reduces to the classical Erd\H{o}s-R\'{e}nyi random graph model, which we denote by $\mathscr{G}(n,p)$. In this paper, it is assumed that $\pi$ and $f$ are known,  and that  $f_1,\ldots,f_Q>0$; for estimating these quantities see, for example, \cite{ airoldi, latoucherobin, olhedewolfe}.

For a fixed graph $G$, it is known (see, for example, \cite{bhj92}, Theorem 5.B)  that the distribution of the number copies of $G$ in the $\mathscr{G}(n,p)$ model is well approximated by an appropriate Poisson distribution if $G$ is a member of the class of strictly balanced graphs (defined below) as long as $p$ is not too large.  In fact, \cite{bhj92} give explicit bounds on the rate of convergence in the Poisson approximation.   

In this paper, we consider a generalisation of the Poisson approximation to the stochastic block model.  We obtain explicit bounds for the rate of convergence, and consider both the cases that the edge  probabilities $\pi_{a,b}$ are constant and that they are themselves random variables supported on a subset of $[0,1]$.  This paper therefore contains the following two features that have not appeared together before:
the random graphs may have local dependence or inhomogeneity, or both, and
the approximation is quantitative (as opposed to only asymptotic).
The second feature has appeared without the first
(for example in \cite{bhj92}) and the first feature has appeared without the
second (for example in \cite{bollobasetal}), at least for the case of cycles in the case of constant average
degree stochastic block models. It is the combination of both features that is novel.

When the edge probabilities are themselves random variables then we assume that they are only locally dependent, in the sense that each edge has a relatively small number of other edges so that their random edge probabilities are not independent. As an example the vertices may have some exogeneous characteristics such as geographical location which influence the probability of an edge to exist, but only locally. 
 
   The latter case is related to a graphon model, where edge probabilities only depend on those edge probabilities where the edges share a vertex. A graphon is represented by a measureable function $h: [0,1]^2 \rightarrow [0,1]$. A graphon model constructs a random graph on $n$ vertices  by assigning independent $U[0,1]$ variables to each vertex. Conditional on these uniform random variables, all edges are independent, and the probability of an edge between vertices $u$ and $v$ is given by $h(U_u, U_v)$. These graphs appear as limits of exchangeable graphs; see, for example, \cite{aldous, diaconis, lovasz}. They are a special case of inhomogeneous random graph models as considered in \cite{bollobasetal}.    

Setting the scene for counting copies of graphs $G$, let $K_n$ be the complete graph with $n$ edges and $\binom{n}{2}$ edges.  Let $G\subset K_n$ be a fixed graph with $v(G)$ vertices and $e(G)$ edges; let $V(G)$ denote the vertex set and $E(G)$ its edge set.  To avoid trivialities, we assume that $e(G)> 1$ and that $G$ has no isolated vertices. We shall be particularly interested in the case that $G$ is a member of the class of strictly balanced graphs, which we now define according to \cite{bhj92}.  Let
\begin{equation*}
d(G) = \frac{e(G)}{v(G)}.
\end{equation*}
Then the graph $G$ is said to be \emph{strictly balanced} if $d(H)<d(G)$ for all subgraphs $H\subsetneq G$. 

Let $\Gamma$ denote the set of $v(G)$-tuples of elements from $\{1,\ldots,n\}$.  Then, $\alpha\in\Gamma$ is a possible position for the subgraph $G$, and  there are $\binom{n}{v(G)}$ such positions.  
 To account for re-labelling of vertices,  let $R_{\alpha}(G)$ denote the set of {{all subgraphs of the complete graph on the $v(G)$-tuple $\alpha$ which are isomorphic to $G$}} (a similar notation was introduced in Picard et al$.$ \cite{picard}).  For any $\alpha\in\Gamma$, the number of elements in the set $R_\alpha(G)$ is given by
\begin{equation} \label{rho}
 \rho(G)=\frac{(v(G))!}{a(G)},
 \end{equation} 
  where $a(G)$ is the number of elements in the automorphism group of $G$. 

Now, let $\mathscr{G}=(\mathcal{V},\mathcal{E})$ be a random graph on $n$ edges. For $\alpha\in\Gamma$ and $G'\in R_{\alpha}(G)$,  let $X_\alpha(G')$ be the indicator random variable for the occurrence, at the $v(G)$-tuple $\alpha$, of a subgraph $G'$ which is  isomorphic to $G$.  We shall let $W$ denote the total number of copies of $G$ in the random graph $\mathscr{G}$,
\begin{equation}\label{weqn}W=\sum_{\alpha\in\Gamma}\sum_{G'\in R_{\alpha}(G)}X_\alpha(G').
\end{equation}

Here, copies are counted as opposed to induced copies where not only all edges of the graph have to appear, but also no edge which is not in the graph is allowed to appear in the copy. For example, the complete graph $K_n$, $n\geq3$,  contains $(n-1)!/2$ copies, but no induced copy, of an $n$-cycle.

To illustrate our notation, consider counting the number of isomorphic copies of the path on three vertices, denoted by $G$, in a graph $\mathscr{G}$ with vertex set $\{1,2,3,4\}$.  We first construct the vertex set $\Gamma$ of all 3-tuples from the set $\{1,2,3,4\}$.  For the set $\{1,2,3\}\in \Gamma$, we consider the indicators $X_{\{1,2,3\}}(G')$, where $G'\in R_{\{1,2,3\}}(G)$.  The set $R_{\{1,2,3\}}(G)$ contains three non-redundant ways $G_1',G_2',G_3'$ that a copy of $G$ can occur on $\{1,2,3\}$, these being if edges  $\{1,2\}$ and $\{1,3\}$ are present; edges  $\{2,1\}$ and $\{2,3\}$ are present; or edges  $\{3,1\}$ and $\{3,2\}$ are present.  We count the number of occurrences of $G$ in $\{1,2,3\}$ using the indicators $X_{\{1,2,3\}}(G_i')$, $i=1,2,3$, and we then repeat this procedure for all $\alpha\in\Gamma$.  Since $|R_{\{1,2,3\}}(G)|=3$ and $|\Gamma|=\binom{4}{3}=6$, there can be at most 18 copies of $G$ in the graph $\mathscr{G}$.  For example, if $\mathscr{G}$ is the circle graph with edge set $\{\{1,2\},\{2,3\},\{3,4\},\{4,1\}\}$, then $X_{\{1,2,3\}}(\{1,2\},\{1,3\})=0$, $X_{\{1,2,3\}}(\{2,1\},\{2,3\})=1$ and $X_{\{1,2,3\}}(\{3,1\},\{3,2\})=0$.

In the stochastic  block model $SBM(n,\pi,f)$, the conditional occurrence probability of an isomorphic copy $G'$ of the subgraph $G$ on $\alpha = (i_1,\ldots,i_{v(G)})$ given the class of each vertex is
\begin{equation*}
\mathbb{P}(X_\alpha(G')=1\,|\,i_1\in c_1,\ldots,i_{v(G)}\in c_{v(G)})=\prod_{1\leq u< v\leq v(G): (u,v) \in E(G) }\pi_{c_u,c_v}.
\end{equation*}  The occurrence probability of an isomorphic copy $G'$ of $G$ is then 
\begin{equation}\label{mueqn}
\mu(G){{=\mathbb{E} X_\alpha(G')}}= \sum_{c_1,c_2,\ldots,c_{v(G)}=1}^{Q} f_{c_1}f_{c_2}
\cdots f_{c_{v(G)}} \prod_{1 \leq u<v \leq v(G): (u,v) \in E(G) }\pi_{c_u,c_v}.
\end{equation}
Note that for any $\alpha,\beta\in \Gamma$ and $G'\in R_\alpha(G),$ $G''\in R_\beta(G)$ we do indeed have $\mathbb{E}X_\alpha(G')=\mathbb{E}X_{\beta}(G'')$.  We therefore have that
\begin{equation}\label{lambdaeqn}\lambda:=\mathbb{E}W=\binom{n}{v(G)} \rho(G) \mu(G).
\end{equation}

In this paper, we use the Stein-Chen method for Poisson approximation, introduced by \cite{chen 0}, to assess the distributional distance between $\mathcal{L}(W)$ and the $Po(\lambda)$ distribution when the fixed graph $G$ is is a member of the class of strictly balanced graphs.  This discrepancy is measured using the total variation distance, which for non-negative, integer-valued random variables $U$ and $V$ is given by
\begin{equation*}d_{TV}(\mathcal{L}(U),\mathcal{L}(V))=\sup_{A\subseteq\mathbb{Z}^+}|\mathbb{P}(U\in A)-\mathbb{P}(V\in A)|.
\end{equation*} 

In deriving  bounds on the total variation distance, we  exploit the local dependence structure of the indicators $X_{\alpha}(G')$.  To this end, for each $\alpha\in \Gamma$, we introduce a set $A_\alpha$ which can be viewed as a dependency neighbourhood of $\alpha$.  In the SBM, as class assignments are independent and the edge probabilities are given, we can take
\begin{equation*}
A_\alpha = \{ \beta \in \Gamma \colon \lvert \alpha \cap \beta \rvert \geq 1 \}.
\end{equation*}
Here, $A_\alpha$ is a dependency neighbourhood of $\alpha$ in the sense that if $\lvert \alpha \cap \beta \rvert = 0 $, then $X_{\alpha}(G')$ and $X_{\beta}(G'')$ are independent for any $G'\in R_\alpha(G)$, $G'' \in R_\beta(G)$.  In Section 3, the edge  probabilities are random variables, in which case finding a suitable dependency neighbourhood $A_\alpha$ is more involved.
With
\begin{equation*}
\eta_{\alpha}(G') = \sum_{\beta \in A_{\alpha}} \sum_{G'' \in R_\alpha(G)}X_\beta(G'')
\end{equation*}
and
\begin{equation}
\label{thetaeqn}\theta_{\alpha}(G') = X_{\alpha}(G')(\eta_{\alpha}(G')-X_{\alpha}(G')),
\end{equation}
 a simple corollary of Theorem 1 in \cite{agg89}, or of Theorem 1.A in \cite{bhj92} is that  
\begin{equation}\label{dtv}
d_{TV}(\mathcal{L}(W),Po(\lambda)) \leq \lambda^{-1}(1-e^{-\lambda}) \sum_{\alpha \in \Gamma} \sum_{G' \in R_\alpha(G)} \big\{\mathbb{E}X_{\alpha}(G') \mathbb{E}\eta_{\alpha}(G')+\mathbb{E}\theta_{\alpha}(G')\big\}.
\end{equation}
Thus bounding the total variation distance between the distribution of the subgraph counts in the SBM and the $Po(\lambda)$ distribution reduces to bounding the expectations on the right-hand side of (\ref{dtv}).  We shall prove our Poisson approximations for subgraph counts (Theorems \ref{thm2.1} and \ref{thm3.1} and Corollary \ref{graphoncor}) using this approach.  The Poisson approximation results of these theorems are valid when the fixed graph $G$ is strictly balanced and the edge  probabilities $\pi_{a,b}$ are not too large.  These theorems generalise Theorem 5.B of \cite{bhj92}, which asserts that a Poisson approximation is valid in the $\mathscr{G}(n,p)$ model under the same conditions.

The Poisson approximation is valid under these conditions in the SBM  for exactly the same reason as it is in the $\mathscr{G}(n,p)$ model: if $G$ is strictly balanced and the $\pi_{a,b}$ are not too large, with high probability the copies of $G$ are vertex disjoint and the $X_{\alpha}(G)$ are close to being independent. Thus, $W$ is the sum of a large number of almost independent indicators with small means, and a Poisson approximation is valid.  In the $\mathscr{G}(n,p)$ model, the Poisson approximation breaks down if $G$ is not strictly balanced \cite{rv85}, although Compound Poisson approximations may still be valid for certain classes of subgraphs; see \cite{stark}.  For this reason, we restrict our attention to strictly balanced graphs.

The rest of the paper is organised as follows.  In Section 2, we use the Stein-Chen method to derive a Poisson approximation for the number of subgraphs in the SBM which are isomorphic to some fixed graph from the class of strictly balanced graphs.  In Section 3, we consider a generalisation of this problem in which the edge probabilities are now (possibly locally dependent) random variables supported on a subset of $[0,1]$.  Again, we derive a Poisson approximation for the number of copies of a fixed subgraph in this model.  Section 4 gives a Poisson approximation of small graph counts in the graphon model. 

\section{Poisson approximation of subgraph counts in the stochastic block model}
In this section, we obtain a Poisson approximation for the number of subgraphs in the SBM which are isomorphic to a fixed graph from the class of strictly balanced graphs.  Before stating this result, we introduce some notation.  Let
\begin{equation}\label{alpha}
\alpha(G) = \min_H \frac{e(G)-e(H)}{v(G)-v(H)}
\end{equation}
and
\begin{equation}\label{gamma} 
\gamma(G) =\min_H(d(G)v(H)-e(H))= \min_H v(H)\cdot(d(G)-d(H)),
\end{equation}
where the minima are taken over all all non-empty subgraphs $H \subsetneq G$ without isolated vertices.  It is worth noting that the graph $G$ is strictly balanced if $\gamma(G)>0$ or $\alpha(G)>d(G)$; see \cite{bhj92}. Also, let 
\begin{equation} \label{pi*}
\pi^*= \max_{1\leq a< b\leq Q} \pi_{a,b}
\end{equation}
denote the maximum edge probability.


\begin{theorem}\label{thm2.1} Suppose that $G$ is a strictly balanced graph.  Then, with the  notation \eqref{mueqn}, \eqref{rho}, \eqref{alpha}, \eqref{gamma} and \eqref{pi*},
\begin{align}
\label{ermg1111thm}d_{TV}(\mathcal{L}(W),Po(\lambda))  &\leq  (1-e^{-\lambda})\rho(G) \bigg\{2 \frac{v(G)^2}{v(G)!}n^{v(G)-1} (\pi^*)^{e(G)} + \pi^*
\nonumber \\
 &\quad+ \sum_{s=2}^{v(G)-1} \binom{v(G)}{s}  \frac{n^{v(G)-s}(\pi^*)^{\kappa(G,s)}}{(v(G)-s)!} \bigg\},  
\end{align}
where 
\begin{equation} \label{kappa} 
\kappa(G,s)=\max(e(G)-sd(G)+\gamma(G),(v(G)-s)\alpha(G)).
\end{equation} 
\end{theorem}

\begin{proof}
We establish our bound by bounding the right-hand side of inequality (\ref{dtv}), starting with $\sum_{\alpha \in \Gamma} \sum_{G' \in R_\alpha(G)} \mathbb{E}X_{\alpha}(G') \mathbb{E}\eta_{\alpha}(G')$. 
For  the dependence set $A_\alpha = \{ \beta \in \Gamma \colon \lvert \alpha \cap \beta \rvert \geq 1 \}$,
\begin{equation}\label{eqn22}\lvert A_{\alpha} \rvert \leq v(G) \binom{n}{v(G)-1} \leq \frac{v(G)^2}{v(G)!}n^{v(G)-1}.
\end{equation}
It is now clear from \eqref{mueqn} and (\ref{eqn22}) that
\begin{align}\label{ineq1}
\mathbb{E}\eta_{\alpha}(G') &=  \sum_{\beta \in A_{\alpha}} \sum_{G' \in R_\alpha(G)} \mathbb{E}X_\beta(G')  =  \lvert A_{\alpha}\rvert \rho(G) \mu(G)  \nonumber \\
&\leq  \frac{\rho(G)v(G)^2}{v(G)!}n^{v(G)-1} \mu(G). 
\end{align}

The more involved part of the proof, where the assumption of strictly balancedness comes into play, is to
bound the expectation $\mathbb{E}\theta_{\alpha}(G')$ from (\ref{thetaeqn}).  When $\alpha$ and $\beta$ have considerable overlap, then $\mathbb{E}X_{\alpha}(G')X_{\beta}(G'')$ may be large compared to $ \mathbb{E}X_{\alpha}(G')$ - but there are not many $\beta$'s which have considerable overlap with $\alpha$. To take account of the overlap, we partition $A_\alpha$ into sets $\{\Gamma_\alpha^s\}_{1\leq s\leq v(G)}$, where $\Gamma_\alpha^s=\{ \beta \in \Gamma \colon \lvert \alpha \cap \beta \rvert =s \}$.  These sets can be bounded above by
\begin{equation*}
\lvert \Gamma_{\alpha}^s \rvert \leq \binom{v(G)}{s} \binom{n}{v(G)-s} \leq \binom{v(G)}{s} \frac{n^{v(G)-s}}{(v(G)-s)!}.
\end{equation*}
Now, recalling (\ref{thetaeqn}), 
\begin{equation*}
\mathbb{E}\theta_{\alpha}(G') = \sum_{s=1}^{v(G)-1}\sum_{\beta \in \Gamma_{\alpha}^s}
\sum_{G'' \in R_\beta(G)} \mathbb{E}X_{\alpha}(G')X_{\beta}(G'') + \sum_{\substack{G'' \in R_\beta(G) \\ G' \neq G''}} \mathbb{E}X_{\alpha}(G')X_{\alpha}(G'').
\end{equation*}
To bound the expectations in the above expression, 
we consider the cases of different overlap $s$ separately.  

Firstly, for $G'\not=G''$, and for $s=v(G)$, so that $\alpha = \beta$, there must be at least 1 edge present in $G''$ which is not in $G'$.
Due to the conditional independence of the edges, for any edge indicator $Y_{i,j} $ which is not included in $X_\alpha (G')$, 
\begin{equation}\label{cond}
\mathbb{P}(Y_{i,j}=1 | X_\alpha(G') = 1) = \sum_{a,b=1}^Q \pi_{a,b} \mathbb{P}(i \in a, j \in b | X_\alpha(G') = 1) \le \pi^*. 
\end{equation}
Hence 
\begin{equation*}
\mathbb{E}X_{\alpha}(G')X_{\beta}(G'') \leq  \mu(G) \pi^* \quad\mbox{ for } \beta \in \Gamma_\alpha^{v(G)}.
\end{equation*}

Next, we consider the case $s=1$, in which $\alpha$ and $\beta$ only intersect at a single vertex.  As a result, $G'$ and $G''$ cannot share an edge.
Using the generalisation of \eqref{cond} that for any set of edges $A$ which does not overlap with the edges in $X_\alpha(G')$,  
\begin{equation}\label{condgen}
\mathbb{P}(Y_{i,j}=1, (i,j) \in A | X_\alpha(G') = 1) \le (\pi^*)^{|A|}, 
\end{equation}
it follows that 
\begin{equation*}
\mathbb{E}X_{\alpha}(G')X_{\beta}(G'') \le \mu(G) (\pi^*)^{e(G)}\quad \mbox{ for } \beta \in \Gamma_\alpha^1. 
\end{equation*}

Finally, we consider the case $2 \leq s \leq v(m)-1$.  We shall derive two bounds for the expectation  $\mathbb{E}X_{\alpha}(G')X_{\beta}(G'')$. 

There are $e(G)$ edges from the subgraph $G'$ given on $\alpha$ and we now consider the number of additional edges resulting from the subgraph $G''$ given on $\beta$. 
Here the underlying graph is $K_n$, the complete graph. 
Consider the subgraph $H$ {{of the intersection graph of $ G' $ and $ G''$}} induced on the intersection of $\alpha$ and $\beta$, which has vertex set $V(H)=\alpha\cap\beta$ and edge set $E(H)$, for which $e\in E(H)$ if and only if $e\in E(G')\cap E(G'')$. Due to the fact that $\lvert \alpha \cap \beta \rvert = s$, we have $v(H)=s$, and, because $G'$ is strictly balanced, it must be the case that $d(H) < d(G)$ (we have $d(G')=d(G)$, as $G$ and $G'$ are isomorphic), and so $e(H)<sd(G)$.  Recalling \eqref{gamma}, we have $e(H)+\gamma(G) \leq sd(G)$, that is $e(H) \leq sd(G)-\gamma(G)$.  Thus, there are at least $e(G)-(sd(G)-\gamma(G)) = e(G)-sd(G)+\gamma(G)$ edges from $G''$ which are not in the subgraph $G'$, and so 
 the union graph of $ G' $ and $ G''$ on $\alpha \cup \beta$ 
has at least $2e(G)-sd(G)+\gamma(G) $ edges.

Alternatively, with $\alpha(G)$ as in \eqref{alpha},
\begin{align}
e(G)-e(H) & =  (v(G)-v(H))\frac{e(G)-e(H)}{v(G)-v(H)} \nonumber \\
& \geq  (v(G)-v(H)) \alpha(G) \nonumber \\
& =  (v(G)-s) \alpha(G), \nonumber
\end{align}
and therefore there are at least $e(G)+(v(G)-s) \alpha(G)$ edges in  the union graph of $ G' $ and $ G''$ on $\alpha \cup \beta$.  This bound in connection with \eqref{condgen} leads to the bound
\begin{equation*}
\mathbb{E}X_{\alpha}(G')X_{\beta}(G'') \leq \mu(G) (\pi^*)^{\kappa(G,s)}\quad \mbox{ for } \beta \in \Gamma_\alpha^s ,
\end{equation*}
where $\kappa(G,s)=\max(e(G)-sd(G)+\gamma(G),(v(G)-s)\alpha(G))$.
Collecting the bounds  gives
\begin{align}
\mathbb{E}\theta_{\alpha}(G') & \leq  \mu(G)  \left\{ \sum_{\substack{G'' \in R_\beta(G) \\ G' \neq G''}} \pi^*+\sum_{\substack{\beta \in \Gamma_{\alpha}^1 \\ G'' \in R_\beta(G)}} (\pi^*)^{e(G)} + \sum_{s=2}^{v(G)-1}
\sum_{\substack{\beta \in \Gamma_{\alpha}^s \\ G'' \in R_\beta(G)}}(\pi^*)^{\kappa(G,s)} \right\}  \nonumber  \\
& \leq  \rho(G) \mu(G)  \bigg\{ \pi^* +\frac{v(G)^2}{v(G)!}n^{v(G)-1} (\pi^*)^{e(G)}\nonumber\\
\label{ineq3}&\quad+\sum_{s=2}^{v(G)-1} \binom{v(G)}{s} \frac{n^{v(G)-s}(\pi^*)^{\kappa(G,s)}}{(v(G)-s)!}\bigg\}.
\end{align}
Finally, substituting  (\ref{ineq1}) and (\ref{ineq3}) into (\ref{dtv}) {{and recalling \eqref{lambdaeqn}}} yields (\ref{ermg1111thm}).
\end{proof}

\begin{remark}
\begin{enumerate}
 \item The stochastic block model structure enters the proof only through the expression for $\mu(G)$ as well as the bound \eqref{condgen}. 
 
 \item Theorem \ref{thm2.1} generalises Theorem 5.B of \cite{bhj92} for the Erd\H{o}s-R\'{e}nyi random graph model to the Stochastic block model.  When we take $\pi_{a,b}=p$ for all $a$, $b$ we recover the same rate of convergence as that  given by Theorem 5.B of \cite{bhj92}.  Indeed the graph combinatorics arguments in our proof are strongly related to those in the proof of Theorem 5.B of \cite{bhj92}. It should, however, be noted that our proof uses a local coupling approach whereas the proof in \cite{bhj92} uses size bias couplings. 

\item To assess the behaviour of the bound it may be advantageous to use the bound 
 $1 - e^{-\lambda} \le \min(1, \lambda)$. Heuristically, a Poisson approximation should hold when $\mu(G)$ is small. When $\mu$ is so small that $\lambda < 1$ then the factor $1 - e^{-\lambda}$ is beneficial. 
 \item  
For a strictly balanced graph, 
  \begin{equation} \label{kappaineq} \kappa(G,s) \ge (v(G)-s) \alpha(G) > (v(G)-s) d(G)
  \end{equation} for all $s =0, \ldots, v(G)-1$. Let 
  $\Delta \kappa = \kappa(G,s)  -  (v(G)-s) d(G)$.  Then $\Delta \kappa > 0$. 
 Using \eqref{condgen} we can bound $\mu(G) \le (\pi^*)^{e(G)}$. 
 If $n (\pi^*)^{d(G)} $ is bounded by $c$  as $n \rightarrow \infty$ then
  $\lambda \le \frac{\rho(G)}{v(G)!} c^{v(G)}$ and 
  $n^{v(G)-s} (\pi^*)^{\kappa(G,s)} \le c^{v(G)-s} (\pi^*)^{\Delta \kappa}$  . Moreover the bound in Theorem \ref{thm2.1} is then of order $O(\min(n^{-1}, n^{-\frac{\Delta \kappa}{d(G)}}))$ as $n \rightarrow \infty$, with proportion vector $f$ and graph $G$ fixed.  

 \item Theorem \ref{thm2.1} is not an asymptotic result but an explicit bound, which may or may not be small.

\item The  result of Theorem \ref{thm2.1} is perhaps most interesting when the limiting $Po(\lambda)$ distribution is non-degenerate in the limit $n\rightarrow\infty$. Suppose that there exist universal constants $c$ and $C$ such that $cn^{-1/d(G)}\leq \pi_{a,b}\leq Cn^{-1/d(G)}$ for all $a,b$.  Then using the inequality 
 $\frac{m^k}{k^k}\leq\binom{m}{k}\leq\frac{m^k}{k!}$, $1\leq k\leq m$  and \eqref{lambdaeqn} we obtain 
\begin{equation*}\frac{\rho(G)}{v(G)^{v(G)}}c^{e(G)}\leq\lambda\leq\frac{\rho(G)}{v(G)!}C^{e(G)}.
\end{equation*} 
Moreover,  
\begin{align} \label{corpn}
d_{TV}(\mathcal{L}(W),Po(\lambda)) &\leq  \min\left(1, \frac{\rho(G)}{v(G)!}C^{e(G)} \right) \rho(G)  \bigg\{\frac{2v(G)^2}{v(G)!}C^{e(G)}n^{-1}\nonumber\\  
&\quad+ Cn^{-1/d(G)} +\min(A,B) \bigg\}, 
\end{align}
where
\begin{eqnarray*}A&=&(1+C^{\alpha(G)})^{v(G)-1}n^{1-\alpha(G)/d(G)};\\
B&=&C^{e(G)+\gamma(G)}(1+C^{-d(G)})^{v(G)-1}n^{-\gamma(G)/d(G)}.
\end{eqnarray*}
\end{enumerate} 
\end{remark} 



\begin{example}\emph{We now use \eqref{corpn} to obtain Poisson approximations for the number of copies of the following fixed graphs with $v\geq3$ vertices in the $SBM(n,\pi,f)$ model.  We consider the following strictly balanced graphs {{on $v$ vertices each}}:}

\begin{enumerate}
\item[$G_{1,v}$] \emph{a tree on the $v$ vertices, with $v-1$ edges;}

\item[$G_{2,v}$] \emph{the cycle graph on the $v$ vertices (with $v$ edges);}

\item[$G_{3,v}$] \emph{the complete graph on $v$ vertices with one edge removed;}

\item[$G_{4,v}$] \emph{$K_v$, the complete graph on $v$ vertices. }
\end{enumerate}

\end{example}

In order to apply \eqref{corpn}, we must compute the quantities $d(G)$, $\alpha(G)$ and $\gamma(G)$ for each graph $G$.  These quantities are easy to compute, and the values are given in Table \ref{dag}.  If for a given graph $G$  there exist universal constants $c$ and $C$ such that $cn^{-1/d(G)}\leq \pi_{a,b}\leq Cn^{-1/d(G)}$ for all $a,b$, then a bound for the total variation distance between the distribution of $W$ and the $Po(\lambda)$ distribution now follows directly from  \eqref{corpn}.  In Table \ref{secondtable}, for each graph $G$, we give the resulting {{bounds on the}} rate of convergence in terms of $n$.  For this rate of convergence it is assumed that the proportion vector $f= f(n) $ remains  constant as $n \rightarrow \infty$,  and that $G$ does not change with $n$. We also give a scaling of the edge probabilities that is required to given a non-degenerate $\lambda$ in the limit.  This scaling is given in terms of $\pi^*=\max_{1\leq a<b\leq Q}\pi_{a,b}$ (note that all the $\pi_{a,b}$ are of the same order). {{ Table \ref{secondtable} shows that the bound on the rate of convergence for the tree graph may be considerably larger than the bound on the rate of convergence in the cycle graph.}}

\begin{table}[ht]
\caption{Values of $d(G)$, $\alpha(G)$ and $\gamma(G)$} 
\centering

\begin{tabular}{ |p{1.6cm}||p{1.6cm}|p{3.1cm}|p{4.6cm}|  }
 \hline
 Graph $G$ & $d(G)$ &$\alpha(G)$&$\gamma(G)$\\
  \hline
 $G_{1,v}$   & $\frac{v-1}{v}$    & $\frac{(v-1)-1}{v-2}=1$ & $\frac{(v-1)^2}{v}-(v-2)=\frac{1}{v}$ \vspace{1mm} \\
 $G_{2,v}$ & 1 & $\frac{v-1}{v-2}$ & $1$ \vspace{1mm}\\
 $G_{3,v}$   & $\frac{(v+1)(v-2)}{2v}$ & $\frac{\binom{v}{2}-1-1}{v-2}=\frac{v^2-v-4}{2(v-2)}$ &  $1/3$ if $v=3$  and \\
 &&& 
 $\frac{(v+1)(v-2)}{2}-\left(\binom{v}{2}-2\right)=1$ \\
 &&& if $v\geq4$ \\
 $G_{4,v}$  &   $\frac{v-1}{2}$  & $\frac{\binom{v}{2}-1}{v-2}=\frac{v+1}{2}$ & $\frac{(v-1)v}{2}-\left(\binom{v}{2}-1\right)=1$\vspace{1mm}\\
 \hline
\end{tabular}
\label{dag}
\end{table}

\begin{table}[ht]
\caption{{{Scaling and bounds on the rate}} of convergence} 
\centering
\begin{tabular}{ |p{1.3cm}||p{3.8cm}|p{6.8cm}|  }
 \hline
 Graph  & Scaling &$d_{TV}(\mathcal{L}(W),Po(\lambda))$\\
  \hline
 $G_{1,v}$   & $\pi^*=Cn^{-v/(v-1)}$    & $O(n^{-1/(v-1)})=O((\pi^*)^{1/v})$  \vspace{1mm} \\
 $G_{2,v}$ & $\pi^*=Cn^{-1}$ & $O(n^{-1})=O(\pi^*)$  \vspace{1mm}\\
 $G_{3,v}$   & $\pi^*=Cn^{-2v/(v+1)(v-2)}$ & 
$O(n^{-1/2})=O((\pi^*)^{1/3})$ if $v=3$ \\
&& and \\ 
&& $O(n^{-2/(v-1)})=O((\pi^*)^{(v+1)(v-2)/v(v-1)})$ \\
&& if $v\geq4$  \\
 $G_{4,v}$  &   $\pi^*=Cn^{-2/(v-1)}$  & $O(n^{-2/(v-1)})=O(\pi^*)$ \vspace{1mm}\\
 \hline
\end{tabular}

\label{secondtable}
\end{table}

\section{Subgraph counts in graph models with random edge  probabilities}

In this section, we consider a model  in which the edge  probabilities are themselves random variables. Let $I= \{u,v: 1 \le u < v \le n\}$ be the index set of potential edges and for $(u,v) \in I $ let $\Theta_{u,v} = \Theta_{v,u} \in [0,1]$ be random variables; given $\Theta_{u,v} = \theta_{u,v}$ the edge indicator $Y_{u,v}$ is Bernoulli distributed with parameter $\theta_{u,v}$. 
Conditional on the edge probabilities $\{\Theta_{u,v}: (u,v) \in I\}$ the edge indicator variables $\{Y_{u,v}: (u,v) \in I\}$ are assumed to be independent. 

We shall assume a local dependence structure for the edge probabilities: 
 for any $(u,v) \in I$ there is a set $B_{u,v}$ such that for any edge set ${\mathcal{E}}$, the collection of random variables $\{ \Theta_{u,v}: (u,v) \in {\mathcal{E}}\} $ is independent of the collection of random variables 
  $\{ \Theta_{x,y}: (x,y)  \in  \left( \cup_{(u,v) \in {\mathcal{E}} } B_{u,v} \right)^c\}. $
 Moreover, we assume that $B_{u,v}$ is of the form
\[ B_{u,v} = \{ (x, w) \in I: x \in M(u,v), w \in N(u,v)\}.\]
We shall often think of $N(u,v)$ as being a small set compared to $\{1,\ldots,n\}$, whereas $M(u,v)$ could be a large set. 
 We denote the least upper bound on $\{|N({u,v})|, (u,v) \in I\}$ by $g$ so that $$|N(u,v)| \le g$$ for all $(u,v)$. For independent edges, if $u<v$ we take $M(u,v) = \{u\}$ and $N(u,v)= \{ v\}$ so that 
 $B_{u,v} = \{ (u,v)\}$ and $g=1$; for graphon models, we can take $M(u,v) = \{ 1, \ldots, n\} $ and $ N(u,v) = \{ u,v\}$ so that  $ B_{u,v} = \{ (x, w) \in I: w \in \{ u, v\} \}$, and  $g=2$. Other examples could include exogenous covariates such as geographic location; edge random variables could be independent if they are further than a certain geographic distance away from each other.

The dependency structure  is now more involved. For $\alpha = (\alpha_1, \ldots, \alpha_{v(G)})$ let ${\mathcal{E}}(\alpha) =  \{(i,j): i \ne j, i, j \in \{ \alpha_1, \ldots, \alpha_{v(G)}\} \}$ denote the set of edges of the complete graph on $\alpha$. 
Then the set 
\begin{equation}\label{aset2}
A_\alpha = \{ \beta \in \Gamma \colon \lvert {\mathcal{E}}(\beta)  \cap \left( \cup_{(u,v) \in {\mathcal{E}}(\alpha)}  B_{u,v} \right)  \rvert \geq 1 \}
\end{equation}
is a dependency neighbourhood of $\alpha$. In particular, if $\beta \not\in A_\alpha$ then $\{ \Theta_{x,y}: (x,y) \in  {\mathcal{E}}(\beta)\} $ is independent of 
$\{  \Theta_{u,v}: (u,v) \in  {\mathcal{E}}(\alpha)\} $.
We can bound the size of this dependency neighbourhood as follows. For $\beta \in A_\alpha$ at least one of the vertices of $\beta$ is in a set $N(u,v)$ for some $u, v \in {\mathcal{E}}(\alpha)$. Each of these sets $N(u,v)$ has at most $g$ elements. Hence 
\begin{equation}\label{aboundgen}
| A_\alpha| \le g v(G) {n \choose {v(G)-1} }  .
\end{equation} 

For a set of edges ${\bf{\gamma}} =  \{ \gamma_1, \gamma_2, \ldots, \gamma_k \} $ we introduce the notation
$V({\bf{\gamma}})$ for the set of vertices which are endpoints in ${\bf{\gamma}}$, so that $|V(\gamma)| \le 2 | {\bf{\gamma}}|.$  We let 
\begin{equation}\label{nukvs}
\nu_{k, v, s} = \max_{ \stackrel{{\bf{\gamma}} = \{ \gamma_1, \gamma_2, \ldots, \gamma_k\}; {\bf{\delta}} = \{ \delta_1, \beta_2, \ldots, \delta_v\}: }{
{\bf{\gamma}} \cap  {\bf{\delta}} = \emptyset; | V({\bf{\gamma}} )\cap  V({\bf{\delta}})| = s
} }
\mathbb{P}\left( \prod_{i=1}^k Y_{\gamma_i} = 1 \Big| \prod_{j=1}^v Y_{\delta_v} = 1\right).
\end{equation} 
With 
$\mu(G) =\mathbb{E} X_\alpha(G') $  and 
\begin{equation*}\label{elambdaeqn}\lambda:=\mathbb{E}W=\binom{n}{v(G)} \rho(G)  \mu(G)
\end{equation*}
we obtain the following variant of Theorem \ref{thm2.1}.

\begin{theorem}\label{thm3.1} Assume that the $\pi_{a,b}$ are arbitrary random variables supported on a subset of $[0,1]$.  Let
{{$\nu_{k,v,s}$ be as in \eqref{nukvs}.}} 
 Suppose that $G$ is a strictly balanced graph.  Then
\begin{align}
d_{TV}(\mathcal{L}(W),Po(\lambda)) &\leq (1-e^{-\lambda})\rho(G) g  \bigg\{2\frac{v(G)^2}{v(G)!}n^{v(G)-1}
\nu_{e(G), e(G), 1} + \nu_{1, e(G), 1}  \nonumber \\ 
\label{ermg1111thm31}&\quad+  \sum_{s=2}^{v(G)-1} \binom{v(G)}{s}  \frac{n^{v(G)-s}\nu_{\kappa(G,s), e(G), s}}{(v(G)-s)!} \bigg\},
\end{align}
where $\kappa(G,s)$ is as in Theorem \ref{thm2.1}. 
\end{theorem}

\begin{proof}The proof proceeds almost exactly as that of Theorem \ref{thm2.1}.  The combinatorial arguments are exactly as before, although note the additional factor of $g$ in (\ref{aboundgen}).  We also deal with the expectations in the formulas for $\mathbb{E}\eta_\alpha(G')$ similarly. A complication arises from bounding the expressions  $\mathbb{E}X_{\alpha }(G') X_\beta( G'')$ which occur in  $\mathbb{E}\theta_\alpha(G')$; the analog of \eqref{condgen} is that 
 for any set of edges $A$ such that $| v(A) \cap v(G') | = s$, 
\begin{equation*}
\mathbb{P}(Y_{i,j}=1, (i,j) \in A | X_\alpha(G') = 1) \le \nu_{|A|, e(G), s}.  
\end{equation*}
\end{proof}

\begin{remark}
In the case that the edges are independent and $\nu = \max_\alpha \mathbb{E}(Y_\alpha)$, we find that 
$\nu_{k, v,s } = \nu^k $ does not depend on $v$ or $s$. 
It is now an immediate consequence of Theorem \ref{thm3.1} that 
\begin{align}
d_{TV}(\mathcal{L}(W),Po(\lambda)) &\leq (1-e^{-\lambda})\rho(G) \bigg\{2 \frac{v(G)^2}{v(G)!}n^{v(G)-1}  \nu^{e(G)} + \nu \nonumber \\ 
\label{ermg1111thm31v2}&\quad
 + \sum_{s=2}^{v(G)-1} \binom{v(G)}{s}  \frac{n^{v(G)-s}\nu^{\kappa(G,s)}}{(v(G)-s)!} \bigg\}.
\end{align}
Taking the $\pi_{a,b}$ to be constants in (\ref{ermg1111thm31v2}) yields $\pi^* = \nu$ and recovers the bound (\ref{ermg1111thm}).
\end{remark}

\section{Subgraph counts in a graphon model}\label{graphon} 

The $h$-graphon model uses 
\[\pi_{u,v} = h(U_u, U_v)\]
where $h: [0,1]^2 \rightarrow [0,1]$ is a symmetric, measureable function and $U_a$, $a =1, \ldots, n $, are independent $U[0,1]$ variables which index the graphon; see for example  \cite{airoldi, bickelchen,   latoucherobin, olhedewolfe}, and \cite{ latoucherobin, wolfeolhede} for graphon estimation. 
In this case edges are not independent, but edges which do not share a vertex are independent, and we can choose $M(u,v) = \{1, \ldots, n\}$ and $N(u,v) = \{u,v\}$ so that $g=2$. Hence 

 \begin{align}\label{graphonmu} 
 \mu(G) &= \int_{[0,1]^{v(G)}} d u_1  \cdots du_{v(G)} \prod_{1 \le i < j \le v(G): (i,j) \in E(G)} h(u_i,u_j) . 
     \end{align} 
With
\begin{equation*}\lambda:=\mathbb{E}W=\binom{n}{v(G)} \rho(G)  \mu(G)
\end{equation*}    
the weak dependence structure yields the following corollary of Theorem \ref{thm3.1}. 

\begin{corollary} \label{graphoncor} Let $\pi_{u,v} = h(U_u, U_v)$
where $h: [0,1]^2 \rightarrow [0,1]$ is a symmetric, measurable function and $U_a$, $a =1, \ldots, n $, are independent $U[0,1]$ variables and let  
\[h^* = \max_{u,v} h(u,v).\]
  Suppose that $G$ is a strictly balanced graph.  Then
\begin{align}
d_{TV}(\mathcal{L}(W),Po(\lambda))&\leq 2  (1-e^{-\lambda})\rho(G) \bigg\{2 \frac{v(G)^2}{v(G)!}n^{v(G)-1} (h^*)^{e(G)} + h^*\nonumber\\
\label{cvhjs}&\quad + \sum_{s=2}^{v(G)-1} \binom{v(G)}{s}  \frac{n^{v(G)-s}(h^*)^{\kappa(G,s)}}{(v(G)-s)!} \bigg\}, 
\end{align}
where $\kappa(G,s)$ is given in \eqref{kappa}. 
\end{corollary} 
 
\begin{proof} Due to the conditional independence of the edges, for any edge indicator $Y_{i,j} $ which is not included in $X_\alpha (G')$, 
\begin{align}
\mathbb{P}(Y_{i,j}=1 | X_\alpha(G') = 1)&= \int_{[0,1]^{v(G)}} du_1  \cdots du_{v(G)} \mathbb{P}(Y_{i,j}=1 | U_v = u_v, v \in V(G'))\nonumber  \\
&= \int_{[0,1]^{v(G)}} du_1  \cdots du_{v(G)} h(u_i, u_j) \nonumber \\ & \le  h^*.  \label{wcondone}
\end{align}
Hence
\begin{equation}\label{wcondgen}
\mathbb{P}(Y_{i,j}=1, (i,j) \in A | X_\alpha(G') = 1) \le (h^*)^{|A|}, 
\end{equation} 
and so $\nu_{k,v,s}\leq (h^*)^k$ for all $v$ and $s$.  Also, $g=2$ for graphon models.  The bound (\ref{cvhjs}) now follows from applying bound (\ref{ermg1111thm31}) of Theorem \ref{thm3.1}.
\end{proof} 

 \begin{remark}
 In the proof of Corollary \ref{graphoncor} we could have replaced \eqref{wcondone} by 
 \begin{align}
\mathbb{P}(Y_{i,j}=1 | X_\alpha(G') = 1) 
&= \int_{[0,1]^{v(G)}} du_1 \cdots du_{v(G)} h(u_i, u_j) \nonumber \\ & \le  \mathbb{E} \Big[\max_{U_i, U_j: i \ne j \in v(G)} h (U_i, U_j)\Big] .  \label{wcondonerep}
\end{align}
For example, if $h(x,y) = \frac12(x + y)$ then $h^*=1$ whereas, using the order statistic notation,  
$$\mathbb{E} \Big[\max_{U_i, U_j: i \ne j \in v(G)} h (U_i, U_j)\Big] = \frac12 \mathbb{E} (U_{(n)} + U_{(n-1)}) = \frac{2v(G)-1}{2(v(G)+1)} < 1.$$
Similarly, \eqref{wcondgen} could be replaced by 
\begin{equation}\label{wcondgenrep}
\mathbb{P}(Y_{i,j}=1, (i,j) \in A | X_\alpha(G') = 1) \le 
\mathbb{E} \left[\max_{U_i, i \in v(G)}  \prod_{(i,j) \in A} h (U_i, U_j)\right]. 
\end{equation}
While \eqref{wcondonerep} and \eqref{wcondgenrep} would yield numerically smaller bounds, $h^*$ is easier to calculate in applications. 
 \end{remark}

\begin{example}
In analogy to copulas,  where Archimedean copulas have proved a useful concept, consider what can be coined an {\it{Archimedean}} graphon: 
Let $h: [0,1]^2 \rightarrow [0,1]$ be given by $h(x,y) = \psi( \psi^{[-1]} (x)  + \psi^{[-1]} ( y) ) $ where $\psi: [0,\infty) \rightarrow [0,1] $ is a continuous, strictly decreasing function which is convex on the open interval $(0, \infty)$ and $\psi^{[-1]}(x) = \inf\{u: \psi(u) \le x\}$ is its generalised inverse. Using the Williamson transform we can write 
$$\psi(x)= \int_{(x,\infty)} \left( 1 - \frac{x}{t}\right) d F_R(t) = \mathbb{E}\left( 1 - \frac{x}{R}\right)_+, $$
where $F_R$ is the c.d.f$.$ of a non-negative random variable $R$ which has no atom at zero, see for example \cite{mcneilneheslova}. If $\inf\{ x: dF_R(x) > 0\} = a_R$  with $a_R > 0$  then  
$$ h^* \le \sup_{x \ge 0} \psi(x) = \int_{a_R}^\infty \left( 1 - \frac{a_R}{t}\right) d F_R(t) = 1 - a_R \mathbb{E}(R^{-1}). $$
In contrast, 
$\mathbb{E} \left[\min_{U_i, i \in v(G)}  \prod_{(i,j) \in A}
 \psi( \psi^{[-1]} (U_i)  + \psi^{[-1]} ( U_j) ) \right] 
$ as used in \eqref{wcondgenrep} would be more difficult to calculate.  
\end{example} 

The next example illustrates how scaling considerations enter in the distributional bound. 

\begin{example}
Let $h: [0,1]^2 \rightarrow [0,1]$ be given by $h(x,y) = x y $.  In this case,  \eqref{graphonmu} gives that 
\begin{align*}\mu(G) 
&=\int_{[0,1]^{v(G)}} d u_1 \cdots du_{v(G)} \prod_{i \in V(G)} u_i^{\mathrm{deg}_G(i)}
 = \prod_{i \in V(G)} \frac{1}{\mathrm{deg}_G(i)+1}, 
\end{align*} 
where $\mathrm{deg}_G(i)$ is the degree of $i$ in $G$, that is, the number of edges in $E(G)$ which have $i$ as an end point; $1 \le \mathrm{deg}_G (i) \le v(G)-1$. Thus in order to obtain a moderate value of $\lambda$, the graph $G$ has to have a large number of vertices with degrees which typically grow like $n$; such graphs are also called dense graphs. 
In this example, $h^* = 1$ and the bound in Corollary \ref{graphoncor} will be of the order $n^{v(G)}$ if the graph $G$ is fixed. 

If instead we consider the function 
$f_n: [0,1]^2 \rightarrow [0,1]$; $h_n(x,y) = n^{-\frac{1}{d(G)}} x y $ then the limiting Poisson distribution is not-degenerate and as in \eqref{corpn} the bound in Corollary \ref{graphoncor} tends to 0 with $n$ tending to $\infty$. 
\end{example}

  Finally, we note that the $h$-graphon model can be viewed as a stochastic block model  if $h$ is piecewise constant.  If $0 =s_1 < s_2 < \cdots < s_{Q-1} = 1$, where $s_i=\sum_{k=1}^if_k$, is a partition of $[0,1]$ so that $h$ is constant on each rectangle $[s_i, s_{i+1} ) \times [s_j, s_{j+1})$, then we could assign type $i$ to vertex $v$ if $U_v \in [s_i, s_{i+1}) $. The randomness now lies only in the class assignments. In this case we recover Theorem \ref{thm2.1}.

\section*{Acknowledgements}

MC acknowledges support from the Department of Statistics, University of Oxford, for a Summer Studentship. RG and GR acknowledge support from EPSRC grant EP/K032402/1.  We would like to thank the referee for helpful comments and suggestions.


\begin{thebibliography}{99}

\bibitem{airoldi} 
Airoldi, E. M.,  Costa, T. B. and Chan, S. H. Stochastic blockmodel approximation of a graphon: Theory and consistent estimation. \emph{Adv. Neur. In.} $\mathbf{26}$ (2013), pp. 692--700.

\bibitem{aldous} 
Aldous, D. J.  Representations for partially exchangeable arrays of random variables. \emph{J. Multivariate
Anal.} $\mathbf{11}$ (1981), pp. 581--598.

\bibitem{ali14} Ali W., Rito, T., Reinert, G., Sun, F. and Deane, C. M. Alignment-free protein interaction network comparison. \emph{Bioinformatics} $\mathbf{30}$ (2014), pp. i430--i437.

\bibitem{agg89} Arratia, R. Goldstein, L. and Gordon, L. Two Moments Suffice for Poisson Approximations: the Chen-Stein Method.  \emph{Ann. Probab.} $\mathbf{17}$ (1989), pp. 9--25.

\bibitem{bhj92} Barbour, A. D., Holst, L. and Janson, S. \emph{Poisson Approximation}. Oxford University Press, Oxford, 1992.

\bibitem{bickelchen} Bickel, P. and Chen, A.  A non parametric view of network models and
Newman-Girvan and other modularities.  \emph{P. Natl. Acad. Sci. USA} $\mathbf{106}$ (2009),   pp.  21068--21073.

\bibitem{bollobasetal} Bollobas, B., Janson, S. and Riordan, O. The phase transition in inhomogeneous
random graphs. \emph{Random Struct. Algor.} (2007), $\mathbf{31}$, pp. 3--122. 

\bibitem{chen 0} Chen, L. H. Y.  Poisson approximation for dependent trials.  \emph{Ann. Probab.} $\mathbf{3}$ (1975), pp. 534--545.

\bibitem{condonkarp} Condon, A. and Karp, R. M.  Algorithms for graph partitioning on the planted partition model. In \emph{Randomization, Approximation, and Combinatorial Optimization. Algorithms and Techniques} Springer Berlin Heidelberg, (1999), pp. 221-232.  

\bibitem{daubin} Daudin, J. J., Picard, F. and Robin, S. A mixture model for random graphs. \emph{Stat. Comput.} $\mathbf{18}$ (2008), pp. 173--183.

\bibitem{diaconis}
Diaconis, P.  and Janson, S. Graph limits and exchangeable random graphs. \emph{Rendiconti
di Matematica} $\mathbf{28}$ (2008), pp. 33--61. 

\bibitem{frank} 
Frank, O. and Strauss, D. Markov graphs. \emph{J. Am. Stat.  Assoc.}   $\mathbf{81}$ (1986), pp. 832-842.

\bibitem{holland} Holland, P. W., Laskey, K. B. and Leinhardt, S. Stochastic blockmodels: First steps.  \emph{Soc. networks}  $\mathbf{5}$ (1983), pp. 109--137. 

\bibitem{karrernewman} Karrer, B., and Newman, M. E.  Stochastic blockmodels and community structure in networks. \emph{Phys. Rev. E} $\mathbf{83}$(1) (2011), 016107.

\bibitem{latoucherobin} 
 Latouche, P.  and  Robin, S.  Bayesian Model Averaging of Stochastic Block Models to Estimate the Graphon Function and Motif Frequencies in a W-graph Model. arxiv:1310.6150, 2013. 

\bibitem{lovasz}
 Lov\'{a}sz, L. and  Szegedy, B.  Limits of dense graph sequences.  \emph{J. Comb. Theory B} $\mathbf{96}$ (2006), pp. 933--957.
 
 \bibitem{matias} 
 Matias, C. and Robin, S.  Modeling heterogeneity in random graphs through latent space models: a selective review.  \emph{ESAIM: Proceedings and Surveys} $\mathbf{47}$ (2014), pp. 55--74.
 
 \bibitem{mcneilneheslova} McNeil, A. and Ne{s}lehov\'{a}, J.   Multivariate Archimedean Copulas, d-monotone functions and $L_1$-norm symmetric distributions. \emph{Ann. Stat.}  $\mathbf{37}$ (2007), pp. 3059--3097. 
 
 
\bibitem{alon}
Milo, R., Shen-Orr, S., Itzkovitz, S., Kashtan, N., Chklovskii, D. and Alon, U. Network motifs: simple building blocks of complex networks. \emph{Science}  $\mathbf{298}$ (2002), pp. 824-827.


\bibitem{ns01} Nowicki, K. and Snijders, T.   Estimation and prediction for stochastic blockstructures. \emph{J. Am. Stat.  Assoc.} $\mathbf{96}$ (2001), pp. 1077--1087.

\bibitem{olhedewolfe} 
Olhede, S. C.  and Wolfe, P. J.  Network histograms and universality of blockmodel approximation. 
\emph{P. Natl. Acad. Sci. USA}  $\mathbf{111}$ (2014), pp. 14722--14727.


\bibitem{picard} Picard, F., Daudin, J. J., Koskas, M., Schbath, S. and Robin, S.  Assessing the exceptionality of network motifs.  \emph{J. Comput. Biol.} $\mathbf{15}$ (2008), pp. 1--20.

\bibitem{rv85} Ruci\'{n}ski, R. J. and Vince, A.  Balanced graphs and the problem of subgraphs of random graphs.  \emph{Congressus Numerantum} $\mathbf{49}$ (1985), pp. 181--190.

\bibitem{sarajlic13} Sarajli\'{c}, A., Janji\'{c}, V., Stojkovi\'{c}, N., Radak, D. and Pr\v{z}ulj, N.  Network
topology reveals key cardiovascular disease genes. \emph{PLoS ONE} $\mathbf{8}$(8) (2013), e71537.

\bibitem{stark} Stark, D.  Compound Poisson approximation of subgraph counts in random graphs.
\emph{Random Struct. Algor.} $\mathbf{18}$ (2001), pp. 39--60.


\bibitem{wolfeolhede} 
Wolfe, P. J. and Olhede, S. C.  Nonparametric graphon estimation. arXiv:1309.5936, 2013.

\end{thebibliography}
\end{document}